\documentclass{article}

\newcommand{\titel}
{Probability functions in the context of signed involutive meadows}

\usepackage{stmaryrd,amssymb,amsmath,amsthm,url}

\newtheorem{theorem}{Theorem}[subsection]

\newtheorem{definition}[theorem]{Definition}

\newcommand{\midd}{\:|\:}
\newcommand{\e}{\operatorname{{\mathbf a}}}
\newcommand{\sg}{\operatorname{{\mathbf s}}}
\newcommand{\reals}{\ensuremath{\mathbb{R}}}
\newcommand{\M}{\ensuremath{\mathbb{M}}}
\newcommand{\Nat}{\ensuremath{\mathbb{N}}}
\newcommand{\IL}{\ensuremath{\textit{IL}}}
\newcommand{\PFP}{\ensuremath{\textit{PF}_P}}
\newcommand{\PFFWP}{\ensuremath{\textit{PFF}_{W,P}}}
\newcommand{\BA}{\ensuremath{\textit{BA}}}
\newcommand{\Md}{\ensuremath{\textit{Md}}}
\newcommand{\Sign}{\ensuremath{\textit{Sign}}}

\title{\titel}

\author{
	Jan A.\ Bergstra ~\&~ Alban\ Ponse
	\\
\\
  {\small
	  Informatics Institute,
	  University of Amsterdam}\\
	{\small \url{https://staff.fnwi.uva.nl/j.a.bergstra/}
	\qquad
	        \url{https://staff.fnwi.uva.nl/a.ponse/}
	}
}
\date{}

\begin{document}

\maketitle

\thispagestyle{empty}

\begin{abstract}
The Kolmogorov axioms for probability  functions are placed in 
the context of signed meadows. 
A completeness theorem  is stated and proven for the resulting 
equational theory of probability calculus. Elementary definitions of 
probability theory are restated in this framework.
\\[2mm]
\end{abstract}

\section{Introduction}
The Kolmogorov axioms for probability functions may be considered a module that can be 
included in a variety of more or less formalized contexts. We will propose and investigate 
some consequences of these axioms when placed in the context of involutive meadows, 
that is meadows where inverse is an involution
following the terminology of~\cite{BM2015}.\footnote{%
  This paper is a revision of \texttt{arXiv:1307.5173v2}.
  The paper needed to be revised because we found too late that an equation 
  mentioned as an axiom
  before is in fact derivable from the other axioms as stated. This fact 
  called for a significant change in emphasis and
  presentation. In addition a survey has been provided of some options to 
  turn conditional probability into a total operator.}
 
In particular we will discuss an axiomatization of a probability function (PF) 
on a Boolean algebra.
The Boolean algebra serves as an event space, the PF defined on it produces 
elements of (values in) a signed meadow that serve as probabilities. 
Special focus is on the case where values are chosen in the signed meadow of real numbers. 
The following objectives motivate the line of development in this paper.

\begin{enumerate}
\item To develop an approach towards strictly equational reasoning about probability. 
\item To provide a finite loose equational specification of probability functions.
\item To provide a useful completeness result for equational axioms of probability functions.
\item To investigate some total versions of the conditional probability operator.
\item To initiate the development of an application for the theory of signed meadows as 
outlined in~\cite{BBP2013} and~\cite{BBP2015}.
\end{enumerate}

We will produce an axiom system consisting of twenty-six equational axioms covering Boolean 
algebra, meadows, the sign function, and the PF. 
Then we will introduce several derived operators and prove a number of simple facts, 
including Bayes' theorem. 

These axioms constitute a finite equational basis for the class of Boolean 
algebra based, real-valued PFs. 
In other words, the completeness results of \cite{BBP2013,BBP2015} extend to the case 
with Boolean algebra based PFs.
We understand this result to convey that the set of twenty-six axioms is 
complete in a reasonable sense.

The paper is structured as follows: in the remainder of this section we 
discuss the concept of a meadow in more detail and provide a survey of relevant design options.
In Section~\ref{sec:2} we introduce some preliminaries. 
In Section~\ref{sec:3} we provide equational axioms for a PF, and in
Section~\ref{sec:4} we discuss completeness. 
In Section~\ref{sec:6} we consider multi-dimensional probability functions, and
Section~\ref{sec:7} contains some concluding remarks.
In Appendix~\ref{sec:5} we discuss an example of equational probabilistic reasoning.

\subsection{A survey of design options for the inverse of 0}
A meadow is a ring-like structure equipped with an inverse function. 
A ring based meadow expands a ring with a one place inverse function (inversive notation), 
or a two place division function (divisive notation). 
The terms `inversive notation' and `divisive notation' were coined in \cite{BM2011a}.

The key design choice that needs to be made when contemplating a meadow concerns 
the way it handles the inverse of 0. 
In a rather scattered literature on the subject a plurality of different options has been 
developed and studied, though in varying levels of detail. A brief survey of these endeavours 
sets the stage for the plan of this paper. 
The listing below is incomplete, but it contains all proposals for which we have 
been able to find an unambiguous description. As a criterion regarding this judgement we have 
required that 
(i) it must be possible to find out
when a closed expressions written using $0,1,+,-,\cdot,(-)^{-1}$ is considered to have a 
value in the mathematical structure at hand, 
(ii) for two closed expressions both having a value it must be 
possible to determine equality in the same structure, and 
(iii) the relation between inverse and division must be transparent. 
We will distinguish three design options for ring based meadows and 
three design options for non-ring based meadows.
We will first survey design options for non-ring based meadows.

\subsection*{Non-ring based meadows}
Three options for setting the inverse of zero in a non-ring based meadow stand out, each
involving an error value  which fails to meet the requirements of a ring. 
Distinguishing these options is facilitated by making use of a uniform terminology. 
\begin{description}
\item [Natural inverse.] If $0^{-1}$ is equated with an unsigned infinite value, 
often denoted by $\infty$, then  0 is said to have a natural inverse. 
The use of natural inverse in mathematics dates back to Riemann at least. 
Wheels are the prominent instance of meadows with natural inverse, see~\cite{Carlstrom}.

\item [Signed natural inverse.] If the inverse of zero is equated with a signed 
infinite value (written say as $+\infty = \infty$, 
which differs from $-\infty$)
we propose to speak of a signed natural inverse. 
This design choice underlies the transreals and transrationals, see~\cite{Anderson}.

\item [Common inverse.]
If the inverse of zero is equated with an error value $\e$ then, following~\cite{BP16}, 
zero is said to have a common inverse. 
Common meadows are meadows based on common inverse.
The error value $\e$  satisfies $x+\e = x \cdot \e = - \e = \e$ 
and for that reason 
fails to comply with the requirements for a ring ($0\cdot x=0$). 
Moreover, the error value is unique.
\end{description}
Also in the case of natural inverse and signed natural inverse, the 
error value(s) fail to comply with the requirements for a ring ($0\cdot x=0$).

\subsection*{Ring based meadows}
For ring based meadows three options may be distinguished.
\begin{description}
\item [Partial inverse.] The most prominent ring based meadow leaves the inverse of 0 
undefined and considers inverse to be a partial function. 

Working with partial inverse deviates from mathematical practice to the extent that 
questions like whether or not $1/0 = 2/0$ must be taken seriously. 
When dealing with partial inverse there are no semantic questions about it, but the 
choice of a logic of partial functions leaves substantial room for design variation, 
beginning with a choice between three ways of looking at the truth value of say 
$1/0 = 1/0$: is it considered as being true, or as being false in an overarching  
two-valued logic, or as not being true in an overarching logic which is not two-valued. 

\item [Symmetric inverse.] If the meadow is based on a regular ring and the value of $0^{-1}$ 
is taken to be 0, 0 is said to have a symmetric inverse. 
The meadows of~\cite{BBP2013,BBP2015} and 
several preceding papers are ring based meadows with symmetric inverse. Alternatively
this case is referred to as featuring an involutive inverse, and such meadows are referred to as 
involutive meadows.

\item [Non-involutive inverse.] If the inverse of $0^{-1}$ is taken to be different from 0, 
$(x^{-1})^{-1}=x$ cannot hold, 
that is inverse is not an involution, and inverse is said to be non-involutive.  
The non-involutive meadows discussed in~\cite{BM2015} that satisfy $0^{-1}=1$ are ring based 
meadows with an asymmetric inverse. 
If the inverse of $0^{-1}$ is taken to be say 17 or any (rational or real) number different 
from 0 and 1, 0 is said to have an ad hoc non-involutive inverse. 
Ad hoc non-involutive inverses come into play when formalizing the theory of 
fields in first order logic in the presence of a function symbol for either inverse or division 
(or both).
\end{description}

\subsection{Working with involutive ring based meadows}
In this paper we will work exclusively with ring based involutive meadows, 
which will be referred to simply as meadows. 
The motivation for this choice is that it appears to be a most straightforward way to pursue the 
objectives that were listed above.
However, we do not claim that for the purpose of developing an equational approach to probability 
working with ring based meadows is the best option, neither do we claim that among the three 
options for ring based meadows working with a symmetric inverse is best suited to this objective.

\section{Boolean algebras and meadows}
\label{sec:2}
In this section we specify the mathematical context on which our axiomatization 
is based.
In particular, we provide specifications for Boolean algebras (Section~\ref{subsec:2.2}),
and for events and (signed) meadows (Section~\ref{subsec:2.3}).

\subsection{Boolean algebras}
\label{subsec:2.1}
A Boolean algebra $(B,+,\cdot,',1,0)$ may be defined as a system with at 
least two elements such that $\forall x,y,z\in B$ the 
well-know postulates of Boolean algebra are valid.
Because we want to avoid overlap with the operations of a meadow, we will
consider Boolean algebras with notation from propositional logic,
thus consider $(B, \vee,\wedge,\neg,\top,\bot)$ and adopt the axioms in 
Table~\ref{Ba}.
In~\cite{Pad83} it was shown that the axioms in Table~\ref{Ba} constitute an 
equational basis.

\begin{table}
\centering
\hrule
\begin{align}
\label{eq:1}
(x\vee y)\wedge y&= y\\
\label{eq:2}
(x\wedge y)\vee y&= y\\
\label{eq:3}
x\wedge (y\vee z) &= (y\wedge x)\vee (z\wedge x)\\
\label{eq:4}
x\vee (y\wedge z) &= (y\vee x)\wedge (z\vee x)\\
\label{eq:5}
x\wedge \neg x &= \bot\\
\label{eq:6}
x\vee\neg x&=\top
\end{align}
\hrule
\caption{\BA, a self-dual equational basis for Boolean algebras}
\label{Ba}
\end{table}

\subsection{Valuated Boolean algebras and some naming conventions}
\label{subsec:2.2}
A Boolean algebra can be equipped with a valuation $v$ that assigns to its 
elements values in a signed meadow. 

In this paper we will investigate the special case where the valuation function 
of a valuated Boolean algebra
is a probability  function by requiring that the valuation satisfies the 
Kolmogorov axioms for probability  functions cast to the setting of signed 
meadows. 

By way of notational convention we will from now on assume that $E$ (for events) is the 
name of the carrier of a Boolean algebra, and that $V$ (for values) names the carrier of the 
meadow in a valuated Boolean algebra.

\subsection{Events and signed meadows} 
\label{subsec:2.3}
The set of axioms in Table~\ref{Md} specifies the class of meadows. 

In the setting of probability  functions the elements of the underlying 
Boolean algebra are referred to as events.\footnote{%
  Events are closed under $-\vee-$, which represents alternative occurrence 
  and $-\wedge-$, which represents simultaneous occurrence, and under negation.}
We will use ``value'' to refer to an element of a meadow,\footnote{%
  Rational numbers and real numbers are instances of values.}
and a probability  function  is a valuation (from events to the values 
in a signed meadow).\footnote{%
  We will exclude probability functions with negative values, a phenomenon known in 
  non-commutative probability theory,
  leaving the exploration of that kind of generalization to future work.}
 
\begin{table}
\centering
\hrule
\begin{align}
\label{eq:7}
	(x+y)+z &= x + (y + z)\\
\label{eq:8}
	x+y     &= y+x\\
\label{eq:9}
	x+0     &= x\\
\label{eq:10}
	x+(-x)  &= 0\\
\label{eq:11}
	(x \cdot y) \cdot  z &= x \cdot  (y \cdot  z)\\
\label{eq:12}
	x \cdot  y &= y \cdot  x\\
\label{eq:13}
	1\cdot x &= x \\
\label{eq:14}
	x \cdot  (y + z) &= x \cdot  y + x \cdot  z\\
\label{eq:15}
	(x^{-1})^{-1} &= x \\
\label{eq:16}
	x \cdot (x \cdot x^{-1}) &= x
\end{align}
\hrule
\caption{\Md, a set of axioms for meadows}
\label{Md}
\end{table}

An expression of type $E$ is an event expression or an event term, an expression of 
type $V$ is a value expression or equivalently a value term.  
In the signature of a valuated Boolean algebra there is just one notation for a probability 
function, the function symbol $P$.\footnote{%
  In some cases the restriction to a single probability function $P$ is impractical and 
  providing a dedicated sort for such functions brings more flexibility and expressive power. 
  This expansion may be achieved 
  in different ways. }

The axioms in Table~\ref{t:a} specify the sign function $\sg(\_)$, 
which presupposes an ordering $<$ on its domain and is defined by
\[\sg(x) =
\begin{cases}
-1 &\text{ if }x < 0,\\
0 &\text{ if }x = 0,\\
1 &\text{ if }0 < x.
\end{cases}\] 
Before commenting on these axioms, we introduce some abbreviations.
First, we shall further write $x-y$ for $x+(-y)$. Below we define
the conditional expression
$p \lhd q \rhd r$, two notations for a division operator, absolute value,
and orderings, where 
$p, q$ and $r$ range over $V$, the carrier of the 
signed meadow in a valuated Boolean algebra.

\begin{table}[t]
\centering
\hrule
\begin{align}
\label{eq:17}
\sg(1_x)&=1_x\\
\label{eq:18}
\sg(0_x)&=0_x\\
\label{eq:19}
\sg(-1)&=-1\\
\label{eq:20}
\sg(x^{-1})&=\sg(x)\\
\label{eq:21}
\sg(x\cdot y)&=\sg(x)\cdot \sg(y)\\
\label{eq:22}
0_{\sg(x)-\sg(y)}\cdot (\sg(x+y)-\sg(x))&=0
\end{align}
\hrule
\caption{$\Sign$, a set of axioms for the sign operator}
\label{t:a}
\end{table}

\begin{enumerate}
\item $1_p  =_{\texttt{def}} p\cdot p^{-1},$
\item $0_p  =_{\texttt{def}} 1 - 1_p,$
\item $p \lhd q \rhd r =_{\texttt{def}} 1_q \cdot p + 0_q \cdot r,$
\item $\dfrac{p}{q} =_{\texttt{def}} p \cdot q^{-1},$ 
\item $p/q =_{\texttt{def}} \dfrac{p}{q},$ 
\item $|p| =_{\texttt{def}} \sg(p) \cdot p,$
\item $p < q =_{\texttt{def}} \sg(q-p) = 1$, and
\item 
$p \leq q =_{\texttt{def}} \sg(\sg(q-p) +1) = 1$.
\end{enumerate}

In Table~\ref{t:a}, axiom~\eqref{eq:22} is an equational representation of the 
conditional equation
\[\sg(x) = \sg(y)~\to~\sg(x+y)=\sg(x).\]
This can be seen as follows: if $\sg(x)=\sg(y)$, then $0_{\sg(x)-\sg(y)}=1$ and
$(\sg(x+y)-\sg(x))=0$, and if $\sg(x)\ne\sg(y)$ then
$0_{\sg(x)-\sg(y)}=0$.

Together, Tables~\ref{Md} and~\ref{t:a} contain the axioms $\Md+\Sign$ for signed 
meadows (we write $+$ instead of $\cup$).
In~\cite{BBP2013} the following identities were shown to be consequences of
$\Md + \Sign$:
\begin{align*}
\sg(x^2)&=1_x
&\sg(x)^{-1}&=\sg(x)\\
\sg(x^3)&=\sg(x)
&\sg(\sg(x))&=\sg(x)
\end{align*}
We further note that the equivalence
$\sg(\sg(p)+1)=1\iff p  = \sg(p) \cdot p = |p|$
is provable from $\Md + \Sign$ (this follows easily from
Theorem~\ref{CT} below).

We will also consider the subclass of signed \emph{cancellation meadows}.
A cancellation meadow satisfies the Inverse Law (IL) of Table~\ref{IL}.

\begin{table}[b]
\centering
\hrule
\begin{align*}
x\neq 0 ~\longrightarrow~ x\cdot x^{-1} &=1.
\end{align*}
\hrule
\caption{Inverse Law (IL)}
\label{IL}
\end{table}
\section{Signed meadow based probability calculus} 
\label{sec:3}
In Section~\ref{subsec:3.1} we formulate axioms for a probability  function. 
Following the methods of abstract data type specification we will focus 
on axioms in equational form. Then, we
discuss a plurality of versions of the conditional probability operator 
(Section~\ref{subsec:3.2})
and some properties thereof, in particular versions of Bayes' theorem. 
Finally, we consider independent events (Section~\ref{subsec:3.3}).

\subsection{Equational axioms for a probability function}
\label{subsec:3.1}
In Table~\ref{PF} we define the set $\PFP$ of axioms for a probability function.
These axioms represent Kolmogorov's axioms in the context of a Boolean algebra 
(rather than a universe of sets) and a signed meadow
(instead of a field). Axiom~\eqref{eq:25} expresses that the sign of $P(x)$ is nonnegative. 
Axiom~\eqref{eq:26} distributes $P$ over finite unions. In the
absence of an infinitary version of axiom~\eqref{eq:26} we 
consider these axioms to constitute an axiomatization for the restricted 
concept of probability  functions only, 
rather than for probability measures in general.

\begin{table}
\centering
\hrule
\begin{align}
\label{eq:23}
P(\top) &=1\\
\label{eq:24}
P(\bot) &=0\\
\label{eq:25}
P(x) & = |P(x)|\\
\label{eq:26}
P(x \vee y) &= P(x)+P(y) - P(x \wedge y)  
\end{align}
\hrule
\caption{$\PFP$, a set of axioms for a probability function with name $P$}
\label{PF}
\end{table}

In combination with the axioms $\BA+\Md$,
the two axioms~\eqref{eq:24} and~\eqref{eq:26} in Table~\ref{PF}
can be replaced by the
single axiom
\begin{equation}
\label{eq:A}
\tag{$\dagger$}
P(x)=P(x\wedge y)+P(x\wedge \neg y)
\end{equation}
where the expressions $x\wedge y$ and $x\wedge \neg y$ characterize two disjoint 
(mutually exclusive) events:
axiom~\eqref{eq:24} follows from 
$P(x)=P(x\wedge x)+P(x\wedge \neg x)$, 
thus $P(x\wedge \neg x)=P(\bot)=0$, and axiom~\eqref{eq:26} follows from 
\[P(x\vee y)\stackrel{\eqref{eq:A}}=P((x\vee y)\wedge x)+P((x\vee y)\wedge \neg x)
=P(x)+P(y\wedge \neg x)\]
and $P(y)\stackrel{\eqref{eq:A}}=P(y\wedge x)+P(y\wedge\neg x)$,
thus $P(y\wedge\neg x)=P(y)-P(x\wedge y)$.
Conversely, axiom~\eqref{eq:A} follows from~\eqref{eq:24} and~\eqref{eq:26}:
\begin{equation}
\label{eq:B}
\tag{$\ddagger$}
P(x)=P((x\wedge y)\vee(x\wedge \neg y))=P(x\wedge y)+P(x\wedge \neg y)-P(\bot).
\end{equation}

A valuated Boolean algebra equipped with a valuation $P$ in some signed meadow $\M$
with carrier $V$ that satisfies all axioms of 
\[\BA+\Md+\Sign+\PFP\] 
will be called a $K(\M,P)$-structure.

\begin{theorem}[Disjoint event factorization]
\label{DisjEvFact} 
$\BA+\Md+\PFP\vdash P(x) = P(x \wedge y) + P(x \wedge \neg y)$.
\end{theorem}

\begin{proof}This is~\eqref{eq:B}, which is shown above.
\end{proof}

\begin{theorem}[Probability upper bound] 
$\BA+\Md+\Sign+\PFP\vdash P(x) \leq 1$.
\end{theorem}

\begin{proof}
First notice
$1 = P(\top) = P(x \vee \neg x) = (P(x) + P(\neg x)) - P(\bot) = P(x) + P(\neg x)$,
so $P(x) = 1- P(\neg x)$.
Because $P(\neg x)\geq 0$ we conclude $P(x) \leq 1$.
\end{proof}

The following theorem asserts in equational form the conditional equation 
$P(y)= 0 \to P(x \wedge y) = 0$, 
using inversive notation.

\begin{theorem} 
\label{PrZeroPropagation}
$ \BA+\Md+\Sign+\PFP \vdash P(x \wedge y)\cdot P(y)\cdot P(y)^{-1} = P(x \wedge y).$
\end{theorem}

\begin{proof} Let $\phi(u,v) $ be as follows:
\[\displaystyle \phi(u,v) \equiv \Bigl(1 - 
\frac{|u| + |v|}{|u| + |v|}\Bigl) \cdot u.\] 
Now $(\reals_0,\sg)\models \phi(u,v) = 0$, and
using the completeness theorem of~\cite{BBP2015} one obtains
$\Md+\Sign \vdash \phi(u,v) = 0$. Substituting 
$P(y \wedge x)$ for $u$ and $P(y \wedge \neg x)$ for $v$ 
and applying Theorem~\ref{DisjEvFact}, one derives
\begin{align*}
\BA+\Md+\Sign+\PFP \vdash 
0 &=\Bigl(1 - \frac{|P(y \wedge x)| + |P(y \wedge \neg x)|}
    {|P(y \wedge x)| + |P(y \wedge \neg x)|}\Bigl) \cdot P(y \wedge x)\\
&=\Bigl(1 - \frac{P(y)}{P(y )}\Bigl) \cdot P(y \wedge x),
\end{align*}
from which the required result follows immediately. 
\end{proof}

\subsection{Conditional probability as a total operator: four options}
\label{subsec:3.2}
Conditional probability $P(x \midd y)$ of event $x$ relative to event $y$ is 
conventionally understood as a partial function 
of $x$ and $y$, defined only if $P(y)$ is nonzero. The objective of 
developing an equational logic for probability theory 
suggests that total versions of the conditional probability operator ought to 
be contemplated.

Conditional probability defined according to Kolmogorov is written below 
as $P^\star(x\midd y)$, where variables $x$ and $y$ range over $E$, 
and is defined by 
\[
\textup{$P^\star(x\midd y) =_{\texttt{def}}\dfrac{P(x\wedge y)}{P(y)}\lhd P(y) \rhd \uparrow$.}
\]
Here $\uparrow$ denotes that the result is undefined.\footnote{%
  We  assume that in a context of partial functions 
  an identity $t=r$ is valid if either 
  both sides are undefined or both sides are defined and equal. This 
  convention, however,  leaves room for alternative readings of the
  expressions at hand. 
  In particular the definition given for $x \lhd y \rhd z$ implies that 
  whenever $t$ is undefined, so is $t \lhd r \rhd s$. That is not 
  a very plausible feature of the conditional and in the presence of partial 
  operations the conditional operator requires a different definition. These 
  complications are to some extent avoided, or rather made entirely explicit,  
  when working with total functions. The use of the notation $P^\star(-|-)$ 
  instead of the common notation $P(-|-)$ is justified by the fact that 
  unavoidably $P^\star(-|-)$ inherits properties from the equational 
  specification of the functions from which it has been made up. Such 
  properties need not not coincide with what is expected from $P(-|-)$.} 
The key advantage of partial conditional probability is that one does not 
introduce a value for, say $P(x\midd \bot)$ 
which might be subsequently disputed. 

Four ways of making conditional probability defined on all inputs will now be 
distinguished.

\begin{definition}[Zero-totalized conditional probability]~
\label{DEFcpzero}
$P^0(x\midd y) =_{\textup{\texttt{def}}}\dfrac{P(x\wedge y)}{P(y)}$.
\end{definition}

We notice that $P^0(\top \midd \bot)=0$, a choice for which no convincing 
philosophical motivation can be put forward. Two advantages can be put 
forward in favour of $P^0(-|-)$:
the logical simplicity that comes with it being total and the calculational 
simplicity that comes with choosing 0 as a value
for $P^0(x \midd y)$ when $P(y) = 0$. 
The following properties are immediate:
\[
P^0(x\midd x) = \dfrac{P(x)}{P(x)}\quad\text{and}\quad
P(x) = P(x) \cdot P^0(x\midd x).
\]
Moreover we have `joint probability factorization':
\[P(x\wedge y) = P(x \wedge y) \cdot P(y) \cdot P(y)^{-1} = (P(x\wedge y)/P(y) )\cdot P(y)  
=P^0(x\midd y) \cdot P(y),\]
and `total probability':  
\begin{align*}
P(x) &=P(x \wedge y) + P(x \wedge \neg y)
\\
&= 
P(x \wedge y) \cdot P(y) \cdot P(y)^{-1} +  
P(x \wedge \neg y) \cdot P(\neg y) \cdot P(\neg y)^{-1}
\\
&= 
P^0(x \midd y) \cdot P(y) + P^0(x \midd \neg y) \cdot P(\neg y).
\end{align*}
As another illustration of the latter advantage we provide below a proof of 
Bayes' rule in its simplest form. 
 
\begin{definition}[One-totalized conditional probability]~
\label{DEFcpone}
$P^1(x\midd y) =_{\textup{\texttt{def}}}\dfrac{P(x\wedge y)}{P(y)}\lhd P(y)\rhd 1$.
\end{definition}
We will write $x \to y$ for $\neg x \vee y$. The principal advantage of 
one-totalized conditional probability over zero-totalized conditional 
probability is the validity of the following rule, which provides some
intrinsic motivation for this design of conditional probability:
\[ (x \to y) = \top ~\Rightarrow~ P^1(x \midd y) = 1.\] 
If $\alpha \in \{\star,0,1\}$ then the function 
\[P \circ^\alpha y =_{\texttt{def}} \lambda x \in E.P^\alpha(x \midd y)\]
is not a probability function for each $y$. In particular,  
if $P(y)= 0$,  $P \circ^\alpha y$ will fail to comply with either 
$P \circ^\alpha y(\top) = 1$ or with $P \circ^\alpha y(\bot ) = 0$. 
Now $\lambda P \in \mathit{PF}. P \circ^\alpha y$ being the well-known
update operator that goes with some applications of Bayes' theorem, it is a 
reasonable requirement that this very operator becomes total as well.  We 
will introduce two options for conditionalization which achieve this 
requirement.
 
\begin{definition}[Safe conditional probability]~
\label{DEFcpsafe}
$P^s(x\midd y) =_{\textup{\texttt{def}}}
\dfrac{P(x\wedge y)}{P(y)}\lhd P(y)\rhd P(x)$.
\end{definition}
We find that $P \circ^{s}y = P$ if $P(y)=0$, which allows the view that 
$\lambda P.P \circ^{s}y$ is an operator mapping probability functions to 
probability functions for all events $y$, or stated differently that 
$\lambda y.(\lambda P.P \circ^{s}y)$ is a total 
mapping from events to probability function transformations. 
$P \circ^s$ is safe because it enforces no update when an inconsistency is 
observed. 

Yet another way to achieve this property of a conditional update is to return 
an exceptional value, in this case the canonical probability function for an 
atomic event.
An atom in $E$ is an event $a \in E$ which satisfies 
$\texttt{atom}(a)  =_{\texttt{def}} 
\forall x \in E.(x \wedge a = a ~ \texttt{OR} ~ x \wedge a = \bot)$. 
For an atom $a \in E$ the probability function $\texttt{pf}_a$ 
is defined by:
\[\texttt{pf}_a(x) =_{\texttt{def}} \begin{cases}
1 &\text{if $x \wedge a =a$},\\
0 & \text{if $x\wedge a=\bot$}.
\end{cases}
\]
\begin{definition}[Exception raising conditional probability for $a\in E$]~
\label{DEFcpexcep}
\[P^{e/a}(x\midd y) =_{\textup{\texttt{def}}}
\dfrac{P(x\wedge y)}{P(y)}\lhd P(y)\rhd {\textup{\texttt{pf}}_a(x)}.\]
\end{definition}
For $P^0(-|-), P^s(-|-),$ and  $P^{e/a}(-|-)$ we are not aware of earlier 
definitions, whereas $P^1(-|-)$ has been considered by Adams in~\cite{Adams}, 
and in subsequent literature. 
For a survey of conditional logic and conditional probabilities 
we refer to~\cite{Milne}. 

Of particular importance given its ubiquitous use is Bayes' theorem. 
Bayes' theorem takes different forms for different versions of conditional 
probability and in each of these cases it
appears as a consequence of \BA+\Md+\Sign+\PFP.

\begin{theorem}[Versions of Bayes' theorem]
\label{thm:1}
In $\BA+\Md+\Sign+\PFP$ the following equations are derivable:
\begin{align*}
1.&~~
P^0(x\midd y) =\dfrac{P^0(y\midd x) \cdot P(x)}{P(y)}
&(\textit{Bayes' theorem for $P^0(-|-)$}),\\[1mm]
2.&~~
P^1(x\midd y) =\dfrac{P^1(y\midd x) \cdot P(x)}{P(y)}\lhd P(y) \rhd 1 
&(\textit{Bayes' theorem for $P^1(-|-)$}),\\[1mm]
3.&~~
P^s(x\midd y) =\dfrac{P^s(y\midd x) \cdot P(x)}{P(y)}\lhd P(y) \rhd P(x)
&(\textit{Bayes' theorem for $P^s(-|-)$}),\\[1mm]
4.&~~
P^{e/a}(x\midd y) =\dfrac{P^{e/a}(y\midd x) \cdot P(x)}{P(y)}\lhd P(y) \rhd 
\normalfont{\texttt{pf}_a(x)}
&(\textit{Bayes' theorem for $P^{e/a}(-|-)$}).
\end{align*}
\end{theorem}

\begin{proof} 
Version 1 can be shown as follows:
\begin{align*}
P^0(x \midd y) &= \dfrac{P(x \wedge y)}{P(y)}
\\[2mm]
&= \dfrac{P(y \wedge x)}{P(y)}\cdot \dfrac{P(x)}{P(x)}
&&\text{by Theorem~\ref{PrZeroPropagation}}
\\[2mm]
&= \dfrac{P(y \wedge x)}{P(x)}\cdot \dfrac{P(x)}{P(y)}
\\[2mm]
&= \dfrac{P^0(y\midd x) \cdot P(x)}{P(y)}.
\end{align*}
Version 2 can be shown as follows:
\begin{align*}
P^1(x \midd y)
&= \dfrac{P(x\wedge y) }{P(y)}\lhd P(y) \rhd 1 
\\[2mm]
&=  \dfrac{P(y)}{P(y)}\cdot\dfrac{P(x\wedge y)}{P(y)} 
+ \Bigl(1 - \dfrac{P(y)}{P(y)}\Bigl)
\\[2mm]
&=
\dfrac{P(y)}{P(y)}\cdot \dfrac{ \Bigl(\dfrac{P(y\wedge x)\cdot {P(x)}}{P(x)}\Bigl)}{P(y)} 
+ \Bigl(1 - \dfrac{P(y)}{P(y)}\Bigl)
&&\quad\text{by Theorem~\ref{PrZeroPropagation}}
\\[2mm]
&=
\dfrac{P(y)}{P(y)}\cdot \dfrac{ \Bigl(\dfrac{P(y\wedge x)}{P(x)}\lhd P(x)\rhd 1\Bigl) 
\cdot P(x)}{P(y)} 
+ \Bigl(1 - \dfrac{P(y)}{P(y)}\Bigl)
\\[2mm]
&= 
\dfrac{P(y)}{P(y)}\cdot\dfrac{P^1(y\midd x) \cdot P(x)}{P(y)} + 
\Bigl(1 - \dfrac{P(y)}{P(y)}\Bigl) 
\\[2mm]
&=\dfrac{P^1(y\midd x) \cdot P(x)}{P(y)}\lhd P(y) \rhd 1.
\end{align*}

Versions 3 and 4, thus Bayes' theorem for the cases $P^s(-|-)$ and for $P^{e/a}(-|-)$,
involve similar calculations and are left to the reader.
\end{proof}

\subsection{Independence of events}
\label{subsec:3.3}
Given a $K(\M,P)$-structure, two events $x$ and $y$ are said to be independent relative 
to that structure if $P(x \wedge y) = P(x) \cdot P(y)$ is valid.

\begin{theorem}
Events $x$ and $y$ are independent if and only if 
$P^0(x\midd y) = P(x) \cdot P^0(y\midd y)$ 
and equivalently if and only if $P^0(y\midd x) = P(y) \cdot P^0(x\midd x)$.
\end{theorem}

\begin{proof}
If $x$ and $y$ are independent, then 
\[P^0(x\midd y) = P(x \wedge y)/P(y) = (P(x) \cdot P(y))/P(y) 
= P(x) \cdot P^0(y\midd y),\]
and similarly one finds $P^0(y\midd x) = P(y) \cdot P^0(x\midd x)$. 

Conversely, from $P^0(x\midd y) = P(x) \cdot P^0(y\midd y)$ one finds
$P(x \wedge y)/P(y) = P(x) \cdot (P(y)/P(y))$,
so multiplying both sides 
by $P(y)$ yields 
\[P(x \wedge y) \cdot (P(y)/P(y))= P(x) \cdot (P(y)/P(y))\cdot P(y),\]
which implies
$P(x \wedge y) = P(x) \cdot P(y)$ by Theorem~\ref{PrZeroPropagation}.
\end{proof}

\section{Logical aspects of equations for probability functions}
\label{sec:4}
In this section we provide a completeness result for 
$\BA + \Md+\Sign+\PFP$ (Section~\ref{subsec:4.1}) and
discuss the use of a free Boolean algebra as an event space
(Section~\ref{subsec:4.2}).

\subsection{Completeness of {$\BA+\Md+\Sign+\PFP$}}
\label{subsec:4.1}
In \cite{BBP2013} it is shown that $\Md+\Sign$ constitutes a finite basis for the 
equational theory of signed cancellation meadows. Stated differently: 
for each equation $t=r$, if $\Md+\Sign+\PFP+ \IL \models t=r$
then also $\Md+\Sign+\PFP \vdash t=r.$
This fact is understood as a completeness results because a stronger set of axioms 
would necessarily exclude some meadows that are expansions of ordered fields.
In a preceding version of this paper\footnote{%
  \url{http://arxiv.org/abs/1307.5173v1}} 
it was shown that the basis theorem 
extends to the setting with probability functions:
if $\BA+\Md+\Sign+\PFP+ \IL \models t=r$
then also $\BA + \Md+\Sign+\PFP \vdash t=r.$

For the purposes of this paper we prefer to make use of a different completeness result for the 
same equational theory that allows us  to obtain a more intuitively appealing completeness result 
for the axiom system $\BA+\Md+\Sign+\PFP$.
This second completeness result is given in terms of 
validity of equations relative to a single signed meadow rather than 
in an elementary class of structures.

We recall the following result from~\cite[Thm.3.14]{BBP2015}, where we write
$\reals_0$ for the meadow that is the expansion of the field of real numbers $\reals$ with 
total inverse  operator and $0^{-1}=0$, and $(\reals_0,\sg)$ for $\reals_0$ expanded with
the sign function $\sg(\_)$. 

\begin{theorem}
\label{CT}
For an equation $t=r$ in the signature of signed meadows: 
$(\reals_0,\sg) \models t = r$ if and only if $\Md + \Sign \vdash t=r$.
\end{theorem}

The same completeness result works for conditional equations:
\begin{theorem} 
\label{CTtwo}
For a conditional equation $t_1 = r_1 \wedge \ldots \wedge t_n= r_n \to t=r$ in 
the signature of signed meadows: 
$(\reals_0,\sg) \models t_1 = r_1 \wedge \ldots \wedge t_n= r_n \to t=r$ if and only if 
\[\Md + \Sign \vdash t_1 = r_1 \wedge \ldots \wedge t_n= r_n \to t=r.\]
\end{theorem}
\begin{proof}
The only if part is follows from the soundness of equational logic. 
If 
\[(\reals_0,\sg) \models t_1 = r_1 \wedge \ldots \wedge t_n= r_n \to t=r\]
then 
\[(\reals_0,\sg) \models(1 -(t_1 - r_1) \cdot  ( t_1 - r_1)^{-1})\cdot \ldots 
\cdot  (1 -(t_n - r_n) \cdot  (t_n - r_n)^{-1}) \cdot (t-r) = 0.\]
Using the above completeness theorem for equational logic with respect to $\reals_0$, 
this fact (say $\Phi$) is provable from $\Md + \Sign$. 
Now assuming that $t_1 = r_1 \wedge \ldots \wedge t_n= r_n$ one 
finds that 
\[(1-( t_1 - r_1) \cdot  (t_1 - r_1)^{-1})\cdot \ldots\cdot(1 -( t_n - r_n) 
\cdot  (t_n - r_n)^{-1}) = 1\]
and in combination with $\Phi$ it follows that $t-r = 0$. 
\end{proof}

A $K(\reals_0,P)$-structure is a model of $\BA+\Md+\Sign+\PFP$ that contains the meadow 
of signed reals, $(\reals_0,\sg)$, as the domain of its values.
We will write $K(\reals_0,P)$ for the class of $K(\reals_0,P$)-structures.

Theorem~\ref{CT} can be extended to the setting of 
$K(\reals_0,P)$-structures thus obtaining a 
satisfactory completeness result for $\BA+\Md+\Sign+\PFP$.

\begin{theorem} 
\label{Compl:1}
The axiom system $\BA+\Md+\Sign+\PFP$ is sound and complete for the  
equational theory of $K(\reals_0,P)$.\footnote{%
  More generally, $\BA+\Md+\Sign+\PFP$ is sound for the class of  
  $K(\M,P)$-structures with $\M$ a cancellation meadow.}

\end{theorem}

\begin{proof} Soundness is obvious and therefore we will focus on completeness.
First, a valid equation over the sort of Booleans is provable from
\BA\ because there are no other means to arrive at the validity of such an equation
(no functions from meadows to Booleans).

Next, let $t=r$ be a valid equation over the sort of meadows.
Assume $x_1,..., x_k$ are the variables ranging over $E$ that occur in $t=r$, 
and $y_1,..., y_\ell$ are the meadow variables that occur in $t=r$. 

Let $m$ be the number of occurrences of $P$ in the equation $t=r$. 
We list these occurrences in a linear order
as $P(f_1(x_1,...,x_k)),... , P(f_m(x_1,...,x_k))$ with $f_i$ appropriately 
chosen Boolean expressions over the event variables  $x_1,..., x_k$.

We choose new variables $z_1,...,z_m$ ranging over $V$, the domain of $(\reals_0,\sg)$,
and we define $t^{\prime}$ and $r^{\prime}$ by replacing in $t$ and $r$ each 
occurrence of $P(f_i(x_1,...,x_k))$ by $z_i$ for $1 \leq i \leq m$.

For $x_1,...,x_k$ there are $2^k$ different conjunctions $\bigwedge_{j=1}^{k}x_j'$
with $x_j'\in\{x_j,\neg x_j\}$. 
Let $\alpha_j$ be an enumeration of Boolean expressions for these conjunctions 
($1 \leq j \leq 2^k$) and note that
\begin{equation}
\label{eq:alpha}
\tag{$*$}
\BA \vdash \textstyle
\bigvee_{j=1}^{2^k}\alpha_j = \top.
\end{equation}
We choose $2^k$ new variables $u_j$ ($1 \leq j \leq 2^k$) ranging over $V$.
We intend to use these variables for representing the values of 
$P(\alpha_j)$.

Each expression $f_i(x_1,...,x_k)$ for $1 \leq i \leq m$
can be written as a disjunctive 
normal form by means of a disjunction of expressions of the form $\alpha_j$:
\begin{equation}
\label{eq:fi}
\tag{$\mathit{eq}_i$}
f_i(x_1,...,x_k)= \textstyle\bigvee_{j \in H_i} \alpha_j
\end{equation}
for appropriately chosen subsets $H_i$ of $\{1,...,2^k\}$.

We will now collect $ 2^k + m + 1$ equations that together precisely capture the 
relation between the variables $z_i$ and $u_j$ under the assumption that for some PF $P$, 
$z_i = P(f_i(x_1,...,x_k))$ and $u_j = P(\alpha_j)$ in some $K(\reals_0,P)$-structure. 
These equations are:
\begin{itemize}
\item 
$u_j = \sg(u_j) \cdot u_j$ for $1 \leq j \leq 2^k$, 
expressing that the $u_j$ are nonnegative.
\item 
$z_i = \sum_{j \in H_i} u_j$ for $1 \leq i \leq m$, 
expressing the consequences in terms of probabilities of the 
mentioned disjunctive normal form for the expressions $f_i(x_1,...,x_k)$ in 
equations~\eqref{eq:fi}.
\item 
$\sum_{j=1}^{2^k} u_j =1$, expressing the fact that 
$\sum_{j=1}^{2^k}P(\alpha_j) = 1$, which in turn follows from~\eqref{eq:alpha}.
\end{itemize}
We write $\Phi$ for the conjunction of these equations.

\noindent
\textbf{Claim.} 
If $\BA+\Md+\Sign+\PFP \models t= r$ then $(\reals_0,\sg) \models \Phi \to t^\prime = r^\prime$.

\begin{minipage}[t]{0.968\linewidth} 
\emph{Proof of Claim.}
Assume that $\BA+\Md+\Sign+\PFP \models t=r$. Suppose that $\sigma$ is a 
valuation of the variables $z_i,u_j,y_\ell$ such that 
$(\reals_0,\sg),\sigma \models \Phi$.
Now from the interpretations of the $z_i$ and $u_j$ one may construct a Boolean algebra with 
the $\alpha_j$ as generators 
as well as a PF $P$ on that Boolean algebra. 
In this structure $t= r$ must hold by assumption, and then $t^\prime = r^\prime$
follows by substitution of variables for expressions according to the definition of 
$t^\prime$ and $ r^\prime$. Hence $(\reals_0,\sg),\sigma \models t'=r'$.
\hfill\emph{End proof of Claim.}
\end{minipage}
\\
The completeness result follows from this claim. 
If $(\reals_0,\sg) \models \Phi \to t^\prime = r^\prime$ then 
Theorem~\ref{CTtwo} implies that this conditional equation can be proven from $\Md+\Sign$. 
The substitution $\theta$ 
replacing the variables 
$z_i$ by $P(f_i(x_1,...,x_k))$ and $u_j$ by $P(\alpha_j)$ can then be applied.
Because the axioms of $\BA+\Md+\Sign+\PFP$ 
suffice to prove that $\theta(\Phi)$ holds, and $\theta(t^\prime) \equiv t$ and 
$\theta(r^\prime) \equiv r$ with 
$\equiv$ denoting syntactic equivalence, it follows that
$\BA+\Md+\Sign+\PFP\vdash t=r$.
\end{proof}

\subsection{Using free Boolean algebras as event spaces}
\label{subsec:4.2}
For the purpose of reformulating some elementary aspects of probability theory and statistics the 
generality of working with arbitrary Boolean algebras is inessential, at least at this initial 
stage in the development of an equational callus of probabilities. 
For that reason we will now introduce several simplifying assumptions:

\begin{itemize}
\item  A finite set $C$ of constants for events is provided. Elements of $C$ are called 
primitive events. We
will only consider free Boolean algebras generated by the primitive events. 

\item With $\BA_C$ we will denote the equations for Boolean algebra in a signature which is 
expanded with the constants in $C$.

\item The class of models of $\BA_C+\Md+\Sign+\PFP$ with
a free event space over $C$, $(\reals_0,\sg)$ as its meadow of values,  
and a probability function $P$ is denoted $K_C(\reals_0,P)$.  
Different structures in $K_C(\reals_0,P)$ only differ in the 
choice (interpretation) of the probability function $P$.

\end{itemize}
These assumptions correspond to what is needed for the specification of examples of 
probabilistic reasoning.

\begin{theorem} $\Md + \Sign +  \BA_C + \PFP$ is sound and complete for the equations
of type $V$ that are true in all 
structures in $K_C(\reals_0,P)$. 
In other words, for $t$ and $r$ terms of sort $V$: 
\[\Md + \Sign +  \BA_C + \PFP \vdash t = r ~\text{ if and only if }~
K_C(\reals_0,P) \models t=r.\]
\end{theorem}

\begin{proof}
The proof is merely a reformulation of the proof of Theorem~\ref{Compl:1}.
\end{proof}

\section{Multi-dimensional probability functions}
\label{sec:6}
In this section we provide axioms for multi-dimensional PFs (Section~\ref{subsec:6.1}), and
discuss two elementary issues:
a condition for the existence of a particular universal PF (Section~\ref{subsec:6.2}), 
and the relation between the multi-dimensional and the one-dimensional case
(Section~\ref{subsec:6.3}).

\subsection{Equational axioms for a probability function family}
\label{subsec:6.1}
Let $D= \{a_1,\ldots,a_d\}$ be a finite, non-empty  set. 
The elements of $D$ are referred to as dimensions. 
With
\[A^f_D\]
we denote the set of finite non-empty sequences of elements of $D$ in which each 
dimension occurs at most once, and with
$\ell(w)$ we denote the length of $w \in A^f_D$. 
Note that $A^f_D$ is finite.
Elements of $\smash{A^f_D}$ serve as arities of probability functions on a multi-dimensional 
event space of dimension $\ell(w)$. If $\ell(w)>1$, then $w$ is written 
as a comma-separated
sequence, e.g. $\ell(a_1,a_3)=2$ and we write $(a_1,a_3)\in \smash{A^f_D}$.

Given an event space $E$ and a name $P$ for a probability function, an arity family 
for $D$ is a subset $W$ of $\smash{A^f_D}$ that is closed under permutation and  
under taking non-empty subsequences.  
Given an arity family $W$ for $D$, a function family for $W$ consists  of a function 
$P^w:E^{\ell(w)} \to V$ for each arity $w \in W$. 
A function family for dimension set $D$, arity family $W$ and function name $P$ is a 
\emph{probability function family} (PFF) if it satisfies the axioms of Table~\ref{PFF}. 
Because in an arity repetition of dimensions is disallowed, these axioms 
reduce to what we had already in the case of a single dimension.

\begin{table}
\centering
\hrule
\begin{align}
P^{a,v,b,v'}(y_1,x_1,\ldots,x_m,y_2,z_1,\ldots,z_{n}) 
&= P^{b,v,a,v'}(y_2,x_1,\ldots,x_{m},y_1,z_1,\ldots,z_{n})
\label{eq:wperm}
\\\nonumber
\text{for all $a,b \in D$ and $(a,v,b,v') \in W$,}
&\text{ where $v$ and/or $v'$ can be empty ($m=0$ and/or $n=0$)}
\\[2mm]
\label{eq:marg0}
P^{a}(\top) &= 1
\\[2mm]
\label{eq:marg}
P^{a,v}(\top,x_1,\ldots,x_{k+1}) &= P^{v}(x_1,\ldots,x_{k+1})
\\[2mm]
\label{eq:bot}
P^{w}(\bot,\vec x) &=0
\\[2mm]
\label{eq:nonneg}
P^w(y,\vec x)&= | P^w(y,\vec x)|
\\[2mm]
\label{eq:vee}
P^w(y \vee z,\vec x) &= P^w(y,\vec x)+P^w(z,\vec x) -  P^w(y \wedge z,\vec x)
\end{align}
\hrule
\caption{$\PFFWP$, axioms for a PFF with arity family $W$ and name $P$,
where $a\in D$, $k\in\Nat$, $\vec x = x_1,\ldots,x_k$ and $P(y,\vec x)=P(y)$ if $k=0$, 
and $w = (a,u)\in W$ with $\ell(w)=k+1$}
\label{PFF}
\end{table}

\subsection{Existence of a universal probability function}
\label{subsec:6.2}
A subset $W$ of $A^f_D$ may or may not have a maximal element under inclusion. 
If $W$ has a maximal element $\overline{w}$ and if we have a probability function family 
$(P^w)_{w \in W}$ for $W$, then $P^{\overline{w}}$ serves as a universal element for the 
family of probability functions because all other members of it can be found via successive 
application of the axioms~\eqref{eq:wperm} - \eqref{eq:bot}.

As it turns out some PFFs cannot be extended with a universal PF.  
In the notation of our specification of probability families we will state a 
specific result that  may serve as a necessary 
condition for the possibility to extend a PFF with a universal element. 

\begin{theorem} 
\label{ThmBCSH}
Given a set of dimensions $D=\{a,b,c,d\}$, an arity family $W$ for $D$
that satisfies $W\supset\{(b,c),~(b,d),~(a,d),~(a,c)\}$,
and a  PFF $(P^w)_{w \in W}$,
let $t$ be the following term:
\[t = P^{b,c}(y,z) + P^{b,d}(y,u) + P^{a,d}(x,u) - P^{a,c}(x,z) - P^{b}(y) - P^{d}(u).\]
Then, if $W$ has a maximal element, then $-1 \leq t \leq 0$, that is,
the following two inequalities must hold for $G_{W,P}=\BA+\Md+\Sign+\PFFWP$: 
\[
G_{W,P}\vdash t+1=\sg(t+1)\cdot (t+1)
\quad\text{and}\quad
G_{W,P}\vdash -t=\sg(-t)\cdot -t.
\]
\end{theorem}
Clearly if a PFF for $D$ contains all of $P^{b,c},P^{b,d},P^{a,d},P^{a,c}$ and if it 
fails to meet either one of the mentioned
inequalities on $t$, then a universal PF cannot be found for it.

These facts are known as the BCHS (Bell, Clauser, Horne, Shimony) inequalities. 
Both were formulated and shown in a set theoretic framework for probability theory 
in~\cite{Rastall} and~\cite{Fine2}, 
and  a straightforward proof is given in~\cite[Section~9.2]{Muynck},\footnote{%
  From this pair of inequalities one can derive the original Bell inequalities
  from~\cite{Bell}. The key observation of Bell was that quantum mechanics gives rise to the 
  hypothesis that a four-dimensional event space exists in which a family of joint 
  probabilities for at most two dimensions can be found that violates the inequalities
  from the theorem.} 
which we repeat here.
\begin{proof}[Proof of Theorem~\ref{ThmBCSH} (taken from~\cite{Muynck})]
\begin{align}
P^{b,c,d}(y,z,u) 
&= P^{a,b,c,d}(x,y,z,u) + P^{a,b,c,d}(\neg x,y,z,u)\nonumber
\\
&\leq P^{a,c}(x,z) + P^{a,d}(\neg x, u) \nonumber
\\
&= P^{a,c}(x,z) + P^d(u) - P^{a,d}(x,u),
\label{(9.3)}
\\[2mm]
P^{b,c,d}(\neg y,z,u) 
&= P^{a,b,c,d}(x,\neg y, z, u) + P^{a,b,c,d}(\neg x,\neg y, z, u) \nonumber
\\
&\leq P^{a,d}(x,u) + P^{a,c}(\neg x,z) \nonumber
\\
&= P^{a,d}(x,u) + P^{c}(z) - P^{a,c}(x,z),
\label{(9.4)}
\\[2mm]
0 &\leq P^{b,c,d}(y,\neg z,\neg u) \nonumber
\\
&= P^{b,c}(y,\neg z) - P^{b,c,d}(y,\neg z, u) \nonumber
\\
&=
P^{b}(y) - P^{b,c}(y,z) - P^{b,d}(y,u) + P^{b,c,d}(y,z,u).
\label{(9.5)}
\end{align}
Combining \eqref{(9.3)} and \eqref{(9.5)} yields
\begin{equation}
0 \leq P^{b}(y) - P^{b,c}(y,z) - P^{b,d}(y,u) + P^{a,c}(x,z) + 
P^{d}(u) - P^{a,d}(x,u).
\label{(9.6)}
\end{equation}
From \eqref{(9.5)} and the equality
\[-P^{c,d}(z,u) + P^{c,d}(\neg z,\neg u) = 1 - P^{c}(z) - P^{d}(u)\]
it follows that
\begin{align}
0 &\leq P^{b,c,d}(\neg y, \neg z,\neg u) \nonumber
\\
&= P^{c,d}(\neg z, \neg u) - P^{b,c,d}(y, \neg z,\neg u) \nonumber
\\
&=
1 - P^{b}(y) - P^{c}(z) - P^{d}(u) + P^{b,c}(y,z) + P^{b,d}(y,u) + P^{b,c,d}(\neg y, z, u).
\label{(9.7)}
\end{align}
Then from \eqref{(9.4)} and \eqref{(9.7)} we get
\begin{equation}
0 \leq 1 - P^{b}(y) - P^{d}(u) + P^{b,c}(y,z) + P^{b,d}(y,u) + P^{a,d}(x,u) - P^{a,c}(x,z). 
\label{(9.8)}
\end{equation}
Inequalities \eqref{(9.6)} and \eqref{(9.8)} can be combined to give 
the inequalities of the theorem.
\end{proof}
\subsection{The multi-dimensional case and the one-dimensional case}
\label{subsec:6.3}
The one-dimensional case is obtained from the multi-dimensional case by taking for $D$ a 
singleton set, say $D=\{d\}$, and subsequently forgetting the single name involved. 
The more challenging question then arises if the multi-dimensional case is already 
implicit in the one-dimensional case. The objective of this section is to discuss that 
matter in some detail. 

Some additional terminology will be needed. 
A permutation compatible function family for 
function name $P$ and an arity family $W$ is a
function family for $P$ and $W$ that satisfies the permutation axiom~\eqref{eq:wperm} 
in Table~\ref{PFF}. 

Given a compatible function family $F_{P,W}$ for $P$ and $W$ its hull of one-dimensional 
projections is defined as the collection
of functions $Q:E \to V$ each of which can be obtained from some $P^w:E^{\ell(w)} \to V$ in 
$F_{P,W}$ for 
some $w = (d,u_1,\ldots,u_t)  \in W$ by choosing $e_1,\ldots,e_t$ in $E^t$ and by setting 
$Q(x) = P(x,e_1,\ldots,e_t)$.

We may now look at the specification  $\BA+\Md+\Sign+\PFP$ in a different way: it axiomatizes 
the notion of a probability function rather than of a particular structure with a probability 
function. Then one may define a multi-dimensional PFF (for $P$ and $W$) alternatively as a 
compatible function family  such that each function in its one-dimensional hull is a 
probability function according to the axioms of $\BA+\Md+\Sign+\PFP$.

This is as close as we can get in turning the multi-dimensional case into an application of 
the one-dimensional case. 
It is not a formal reduction because that would require that it is guaranteed that for each 
$Q$ the same meadow is used to determine its status as a probability function. 
The latter constraint lies outside the expressive power of first order equational logic, 
however. This leaves us with the following state of affairs: although $\BA+\Md+\Sign+\PFP$ 
axiomatizes probability functions in a satisfactory manner, when it comes to essential 
allocations such as Theorem~\ref{ThmBCSH}, a more general axiom system is needed to take care 
of the multi-dimensional case.

\section{Concluding remarks}
\label{sec:7}
The incentive for this work came from a talk given by professor Ian Evett on the 
occasion of the retirement
of dr.~Huub Hardy as a driving force behind the MSc Forensic Science at the 
University of Amsterdam.\footnote{%
  This meeting took place at Science Park Amsterdam, Friday June 7, 2013 under the 
  heading ``Frontiers of Forensic Science'', and was organized by Andrea Haker.} 
That talk illustrated the headway that the Bayesian approach to reasoning in forensic 
matters has made in recent 
years. However, Evett also highlighted the conceptual and political problems that may 
still lie ahead of its universal adoption in the legal process.

In order to improve the understanding of these issues an elementary 
logical formalization of reasoning with probabilities might be useful.
With that perspective in mind we came to the conclusion
that in spite of the abundance of introductory texts to probability theory,
the development of an axiomatic approach from first 
principles may yet cover new ground. 
The formalization of probabilities in terms of equational logic 
outlined above is intended to serve as a point of departure from which to develop
presentations of probability theory that may be
be helpful when a formal and logically precise perspective on reasoning with 
probabilities is aimed at.

We acknowledge many discussions with Andrea Haker (University of Amsterdam) 
regarding the relevance of logically grounded reasoning methodologies in forensic 
science.

\appendix

\section{A standard example of equational probabilistic reasoning}
\label{sec:5}
In this appendix we analyse a straightforward example of an application of Bayes' 
theorem.
This example has been derived from Example 1.2 as presented in the freely 
accessible 2015-version of \cite{Barber2012}. In Appendix~\ref{subsec:5.1} the
example is described, and in Appendix~\ref{subsec:5.2} we analyse how a simple
modification of this example may turn it into an inconsistent one and draw some conclusions.

\subsection{The example}
\label{subsec:5.1}
We assume the following hypothetical but conceivable data: 
\begin{enumerate}
\item A rare disease RD occurs with probability $1/100,000$ in the population of a 
country CO.
\item A potentially problematic nutritional habit NH is very widespread, in fact 4 
out of 10 people in CO show NH.
\item It has been found that 8 out of 10 persons in CO who are suffering from RD 
show NH as well.
\end{enumerate}
The question is to find the probability that someone showing NH suffers from RD. 
In order to answer that question the formalization
of the three facts is as follows: we assume that \textsl{\textsc{rd}} and \textsl{\textsc{nh}} 
are names for
events, and we assume that the PF $P$ comprises the available probabilistic data: 
$P(\textsl{\textsc{rd}}) = 1/100,000$, $P(\textsl{\textsc{nh}}) = 4/10$, and
$P^0(\textsl{\textsc{nh}}\midd \textsl{\textsc{rd}}) = 8/10$.

Strictly speaking \textsl{\textsc{rd}} and \textsl{\textsc{nh}} are used as event variables 
and the assumptions are viewed as conditions.  
Computing a probability is used as a shorthand for proving that it equals 
some real value.

The question then is to compute $p=P^0(\textsl{\textsc{rd}}\midd \textsl{\textsc{nh}}) $. 
Using Bayes' theorem one finds: 
\[p=\frac{P^0(\textsl{\textsc{nh}} \midd \textsl{\textsc{rd}}) \cdot 
P(\textsl{\textsc{rd}})}{P(\textsl{\textsc{nh}})} 
= \frac{8/10 \cdot 1/100,000}{4/10} = 0.2 \cdot 10^{-4}.\]
Thus, we find for $E=\{P(\textsl{\textsc{rd}}) = {1/100,000}, P(\textsl{\textsc{nh}}) = {4/10},
P^0(\textsl{\textsc{nh}}\midd \textsl{\textsc{rd}}) = {8/10} \}$ that  
\begin{align*}
&\BA+\Md+\Sign+\PFP +E\vdash
P^0(\textsl{\textsc{rd}}\midd \textsl{\textsc{nh}})  = {0.2 \cdot 10^{-4}}.\end{align*}
At face value this result is informative.

\subsection{Modifications of the example}
\label{subsec:5.2}
One may notice that if NH were less widespread, 
say $P(\textsl{\textsc{nh}})=1/500$, the  value of 
$P^0(\textsl{\textsc{rd}}\midd\textsl{\textsc{nh}})$ computed above
changes significantly: in this case we find for 
$q=P^0(\textsl{\textsc{rd}}\midd\textsl{\textsc{nh}})$ that 
\[q=\frac{P^0(\textsl{\textsc{nh}}\midd \textsl{\textsc{rd}})\cdot P(\textsl{\textsc{rd}})}{P(\textsl{\textsc{nh}})} = 
\frac{8/10 \cdot 1/100,000}{1/500} = 0.4 \cdot 10^{-2}.\]
Again, at face value this result is informative.

Now we may consider the case that the occurrence of NH is even more rare, say 
1 out of 1,000,000. For $r=P^0(\textsl{\textsc{rd}}\midd\textsl{\textsc{nh}})$ we find 
in this case
\[r=\frac{P^0(\textsl{\textsc{nh}} \midd \textsl{\textsc{rd}}) \cdot 
P(\textsl{\textsc{rd}})}{P(\textsl{\textsc{nh}})}  = \frac{8/10 \cdot 1/100,000}
{1/1,000,000} = 8.\]
This outcome is of course ``wrong''  (probabilities are 
supposed not to exceed 1).

The interesting aspect of this example and its modifications
is that the first two equations and 
computations (for $p$ and for $q$) correspond 
with conventional textbook examples, while the third variation 
(for $r$) indicates that something 
might have gone wrong in all three cases. We conclude this:
\begin{enumerate}
\item
The production of a value for $P^0(\textsl{\textsc{rd}}\midd \textsl{\textsc{nh}})$ 
that exceeds 1 constitutes a failure of the 
reasoning process at hand.

\item
That failure is caused by an underlying fault (the failure is 
merely a symptom of that fault).

\item
The failure lies in non-detection of the fact that the third set of assumptions is 
incoherent: it represents a specification of a partial function from $E$ to $\reals_0$  
which cannot be extended to a total PF. 
In this case, with $E=\{P(\textsl{\textsc{rd}}) = {10^{-5}}, P(\textsl{\textsc{nh}}) = {10^{-6}},
P^0(\textsl{\textsc{nh}}\midd \textsl{\textsc{rd}}) = {8 \cdot 10^{-1}} \}$:
\begin{align*}
&\BA+\Md+\Sign+\PFP + E \vdash 0=1.
\end{align*}
To see this one may notice that the data from which 
$P^0(\textsl{\textsc{rd}}\midd \textsl{\textsc{nh}})$ has 
been computed allow the following proof:
\[1/1,000,000 = 
P(\textsl{\textsc{nh}}) \geq P(\textsl{\textsc{nh}}\wedge\textsl{\textsc{rd}}) = 
P^0(\textsl{\textsc{nh}}\midd \textsl{\textsc{rd}}) \cdot P(\textsl{\textsc{rd}}) = 
8/10 \cdot 1/100,000\] 
from which one easily obtains 0 = 1. 

\item 
Determining the fault underlying the failure is harder. When providing an example either a 
coherence check should have been be applied to the original data, 
or the risk of getting invalid results must be accepted. 
Both assertions lie outside equational logic proper.
As the consistency check is an NP-complete question in general 
the suggestion to check consistency first may be considered unconvincing. 

\item 
Assuming that data come from valid experimental procedures, the realistic background of 
these data may be understood to provide a justification for not checking consistency in 
advance of further usage. Under that
assumption calculations involving Bayes' theorem such as in the example are justified. 
\end{enumerate}
\end{document}